\documentclass[11pt]{amsart}

\usepackage{amscd,amssymb,amsopn,amsmath,amsthm,mathrsfs,graphics,amsfonts,enumerate,verbatim,calc
}

\usepackage{url}

\usepackage[OT2,OT1]{fontenc}
\newcommand\cyr{%
\renewcommand\rmdefault{wncyr}%
\renewcommand\sfdefault{wncyss}%
\renewcommand\encodingdefault{OT2}%
\normalfont
\selectfont}
\DeclareTextFontCommand{\textcyr}{\cyr} 

\usepackage{amssymb,amsmath}

\DeclareFontFamily{OT1}{rsfs}{}
\DeclareFontShape{OT1}{rsfs}{n}{it}{<-> rsfs10}{}
\DeclareMathAlphabet{\mathscr}{OT1}{rsfs}{n}{it}

\numberwithin{equation}{section}
\newtheorem{theorem}{Theorem}[section]
\newtheorem{lemma}[theorem]{Lemma}
\newtheorem{proposition}[theorem]{Proposition}
\newtheorem{corollary}[theorem]{Corollary}

\newtheorem{Problem}{Problem}

\theoremstyle{definition}
\newtheorem{definition}[theorem]{Definition}
\newtheorem{remark}[theorem]{Remark}

\theoremstyle{remark}

\newtheorem{acknowledgement}{Acknowledgement}

\newcommand{\Spec}{\operatorname{Spec}}

\newcommand{\V}{\operatorname{V}}

\newcommand{\Ext}{\operatorname{Ext}}

\newcommand{\Hom}{\operatorname{Hom}}

\newcommand{\Frac}{\operatorname{Frac}}

\newcommand{\fm}{\frak{m}}
\newcommand{\fp}{\frak{p}}

\newcommand{\fM}{\frak{M}}

\newcommand{\Frob}{\operatorname{Frob}}

\begin{document}
\title[A study of perfectoid rings via Galois cohomology]
{A study of perfectoid rings via Galois cohomology}

\author[R. Kinouchi]{Ryo Kinouchi}
\address{Department of Mathematics, Institute of Science Tokyo, 2-12-1 Ookayama, Meguro, Tokyo 152-8551, Japan}
\email{kinouchi.r.aa@gmail.com}

\author[K. Shimomoto]{Kazuma Shimomoto}
\address{Department of Mathematics, Institute of Science Tokyo, 2-12-1 Ookayama, Meguro, Tokyo 152-8551, Japan}
\email{shimomotokazuma@gmail.com}




\begin{abstract} 
In his foundational study of $p$-adic Hodge theory, Faltings introduced the method of almost \'etale extensions to establish fundamental comparison results of various $p$-adic cohomology theories. Scholze introduced the tilting operations to study algebraic objects arising from $p$-adic Hodge theory in mixed characteristic via the Frobenius map. In this article, we prove a few results which clarify certain ring-theoretic or homological properties of the tilt of an extension between perfectoid rings treated in the construction of big Cohen-Macaulay algebras.
\end{abstract}

\maketitle

\section{Introduction}

Recently, it has become more obvious that perfectoid geometry is a powerful tool for solving certain problems of mixed characteristic in algebraic geometry and commutative algebra. However, one is forced to encounter non-Noetherian rings/spaces as building blocks of perfectoid geometry for which the method of traditional commutative ring theory does not work effectively. The present paper arises from an attempt to understand the breakthrough result \cite{Bh20} and unveil algebraic structure of some important commutative rings, such as $p$-adic period rings which are fundamental to $p$-adic Hodge theory. Now, let us start with a history of $p$-adic rings studied by Faltings in \cite{Fa02} with further extensions done by Andreatta in \cite{An06}. Let $R$ be an \'etale algebra over $V[T_1,1/T_1,\ldots,T_n,1/T_n]$, where $V$ is a complete discrete valuation ring with perfect residue field of mixed characteristic $p>0$. Assume that $R/pR$ is not a zero ring and that $R \otimes_V \V_\infty$ is an integral domain, where $V_\infty$ is a non-discrete valuation domain obtained by adjoining all $p$-power roots of $p$ and $1$ to $V$. Let $R_\infty$ be an integral domain such that  $R_\infty$ is obtained by adoining all $p$-power roots of $T_i$'s to $R_{V_\infty}:=R \otimes_V V_\infty$. Let $T$ be the maximal \'etale extension domain of $R_{V_{\infty}}[1/p]$. Denote by $R_{\infty,p}$ the integral closure of $R$ in $T$. Then we have $R_{V_{\infty}} \subseteq R_\infty \subseteq R_{\infty,p}$. Let $G_\infty$ be the Galois group of $R_{\infty}$ over $R_{V_{\infty}}$ and let $G$ be the Galois group of $R_{\infty,p}$ over $R_{V_{\infty}}$. Then Faltings proved the following fundamental result called the \textit{Almost purity theorem}.

\begin{theorem}[Faltings]
Let the notation be as above. Then $R_{\infty} \to R_{\infty,p}$ is an inductive colimit of $(p)^{1/p^\infty}$-almost finite \'etale extensions and the $p$-adic completion of $R_{\infty,p}$ is perfectoid. Moreover, the natural map of Galois cohomology groups
$$
H^i(G_\infty,R_\infty/pR_\infty) \to H^i(G,R_{\infty,p}/pR_{\infty,p})
$$
is a $(p)^{1/p^\infty}$-almost isomorphism for every $i \ge 0$.
\end{theorem}

This theorem is essential in the proof of Hodge-Tate comparison theorem, in which the ring $R_{\infty,p}$ plays a pivotal role. Its ring-theoretic properties were studied in \cite{An06}. However, the method employed in \cite{An06} relies on a delicate study of ramification theory. In this article, we construct analogues of $R_\infty$ and $R_{\infty,p}$ by taking $R$ to be an unramifed complete regular local ring of mixed characteristic and use the \textit{tilting operation} (see Definition \ref{tiltingring}) from perfectoid geometry to clarify the algebraic structure of $R_{\infty,p}$ and the tilt of the integral extension $R_\infty \to R_{\infty,p}$. For instance, it was proved in \cite{Shi16} that $R_{\infty,p}$ is an almost Cohen-Macaulay algebra. As another motivational background, let us first recall the following remarkable result. We refer the reader to \cite{Bh20}, \cite{HH92}, \cite{HL07} and \cite{Q16} for the statements and proofs.

\begin{theorem}[Bhatt, Hochster-Huneke]
\label{BigCM1}
Let $(R,\fm,k)$ be a excellent local domain of residual characteristic $p>0$. Denote by $R^+$ the absolute integral closure of $R$ and by $\widehat{R^+}$ the $p$-adic completion of $R^+$. Then $\widehat{R^+}$ is a balanced big Cohen-Macaulay $R$-algebra.
\end{theorem}

This result has the following implication in positive characteristic.

\begin{corollary}
Assume that $(R,\fm,k)$ is a complete Noetherian local domain of mixed characteristic with perfect residue field. Let $p,x_2,\ldots,x_d$ be a system of parameters of $R$. Denote by $(R^+)^\flat$ the tilt of $R^+$ and set $p^\flat:=(p,p^{1/p},\ldots),x_2^\flat:=(x_2,x_2^{1/p},\ldots),\ldots,x_d^\flat:=(x_d,x_d^{1/p},\ldots)$ as elements of $(R^+)^\flat$. Then $p^\flat,x_2^\flat,\ldots,x_d^\flat$ is a regular sequence on $(R^+)^\flat$.
\end{corollary}

\begin{proof}
From the basic theory of tilts, we have two exact sequences
$$
0 \to \widehat{R^+} \xrightarrow{p} \widehat{R^+} \to R^+/pR^+ \to 0~\mbox{and}~0 \to (R^+)^\flat \xrightarrow{p^\flat} (R^+)^\flat \to R^+/pR^+ \to 0,
$$
together with an isomorphism $R^+/pR^+ \cong (R^+)^\flat/p^\flat (R^+)^\flat$. Now the claim of the corollary follows from Theorem \ref{BigCM1}.
\end{proof}

Bhatt's strategy for proving Theorem \ref{BigCM1} is to reduce to the corresponding problem in positive characteristic. Namely, he studied the Frobenius structure on the local cohomology module of the perfect $\mathbb{F}_p$-algebra $(R^+)^\flat$. A heuristic observation tells us that it is relatively easier to study (almost) Cohen-Macaulay or other homological properties for the perfect ring $(R^+)^\flat$ than to study the perfectoid ring $\widehat{R^+}$ in mixed characteristic. However, it is not clear what to expect on the intrinsic structure of $\widehat{R^+}$ or $(R^+)^\flat$ other than the Cohen-Macaulay property as proved by Bhatt, due to the fact that these are highly non-Noetherian rings. Indeed, some evidence is given in \cite{G24} that such rings could be be of quite huge (infinite) Krull dimension. Another interesting question is the coherence of these big rings, which is explored in \cite{Patankar1} and \cite{Patankar2}. We obviously have inclusions $R_\infty \subseteq R_{\infty,p} \subseteq R^+$ and there is a nice description of $\widehat{R_\infty}$ or $(R_\infty)^\flat$ thanks to Proposition \ref{basictilting} and Proposition \ref{valuationtilting}. We remark that Proposition \ref{valuationtilting} was already proved by Andreatta in \cite{An06}. Here, we give a new proof using big Cohen-Macaulay algebras constructed by Gabber and Ramero. In any case, it is desirable to find an alternative approach to Theorem \ref{BigCM1} by studying the $p$-adically completed extensions $\widehat{R_\infty} \subseteq \widehat{R_{\infty,p}} \subseteq \widehat{R^+}$ or the tilted extensions $(R_\infty)^\flat \subseteq (R_{\infty,p})^\flat \subseteq (R^+)^\flat$. Therefore, we are naturally led to the following problem.

\begin{Problem}
\label{tiltedRing}
Let $A \to B$ be an integral extension of $p$-torsion free rings such that the $p$-adic completions of $A$ and $B$ are perfectoid. Then is the $p$-adic completion $\widehat{A} \to \widehat{B}$ still close to being integral? More specifically, can $\widehat{R^+}$ or $\widehat{R_{\infty,p}}$ be constructed from a union of distinguished integral extensions of $\widehat{R_{\infty}}$ up to $p$-adic completion? One can ask the similar question for the tilts of these rings.
\end{Problem}

We will not be precise on what distinguished integral extensions mean in the present article. We will prove some results concerning $(R_{\infty,p})^\flat$ and $\widehat{R_{\infty,p}}$ in Theorem \ref{Galoistilting} and Corollary \ref{Galoistilting2}, respectively. Somewhat surprisingly, although the tilting or the $p$-adic completion of an integral extension is not necessarily integral, it will be shown that $(R_\infty)^\flat \to (R_{\infty,p})^\flat$ or $\widehat{R_\infty} \to \widehat{R_{\infty,p}}$ is integral up to completion. We think this is a key step toward gaining a good understanding of $(R_{\infty,p})^\flat$ or $\widehat{R_{\infty,p}}$. Another fruitful idea is perhaps to understand $R_{\infty,p}$ via its tilt using the Riemann-Hilbert correspondence in positive characteristic \cite{BL19}. There is still a huge gap between $R_{\infty,p}$ and $R^+$. We hope to study a finer structure of big rings such as $(R^+)^\flat$ or $\widehat{R^+}$ in a future's occasion.

\begin{acknowledgement}
The authors are grateful to S. Patankar for pointing out errors in an earlier version of this paper, and to the referee for his/her detailed comments and suggestions which greatly improved the presentation of the article.
\end{acknowledgement}

\section{Preliminaries and notation}

To prove the main theorem, we make some preparations. Let us recall the notion of almost integral closure in the form which suffices for our purpose.

\begin{definition}
Let $A$ be a ring with a regular element $t \in A$. An element $x \in A[1/t]$ is \textit{almost integral} over $A$ if there is a fixed $m \in \mathbb{N}$ such that $t^m x^n \in A$ for all $n>0$. The set of elements $x \in A[1/t]$ that are almost integral over $A$ is called the \textit{complete integral closure} of $A$ in $A[1/t]$. We denote this by $A^*_{A[1/t]}$.
\end{definition}

One can show that $A^*_{A[1/t]}$ is indeed a subring of $A[1/t]$, which coincides with the integral closure of $A$ in $A[1/t]$ in the case that $A$ is Noetherian. We will make an essential use of complete integral closure. Fix a \textit{basic setup} $(V,I)$. In other words, $V$ is a commutative ring with an ideal $I \subseteq A$ such that $I=I^2$. See \cite{GR03} for details. We recall the definition of almost \'etale ring maps.

\begin{definition}
Let $(V,I)$ be a basic setup. Then a $V$-algebra map $f:A \to B$ is \textit{$I$-almost finite \'etale} if $B$ is an $I$-almost finitely generated, $I$-almost projective $A$-module, and $B$ is an $I$-almost projective $B \otimes_A B$-module with respect to the diagonal map $\mu:B \otimes_A B \to B$.
\end{definition}

We will need to use almost Galois coverings for which we refer to \cite[\S 1.9]{An18}.

\begin{definition}[Almost $G$-Galois cover]
\label{AlmostGaloisCov}
Let $(V,I)$ be a basic setup and let $f:A \to B$ be a $V$-algebra map. Assume that a finite group $G$ acts $A$-linearly on $B$. We say that $A \to B$ is an \textit{$I$-almost G-Galois cover} if the following conditions are satisfied.
\begin{enumerate}
\item
Let $B^G=\{b \in~|~g(b)=b;~\forall g \in G\}$. Then the induced map $A \to B^G$ is an $I$-almost isomorphism.

\item
The natural map
$$
B \otimes_A B \to \prod_{g \in G} B;~b_1 \otimes b_2 \mapsto (g(b_1)b_2)_{g \in G}
$$
is an $I$-almost isomorphism.
\end{enumerate}
\end{definition}

In \cite[Definition V.12.2]{AGT16}, there is a definition of almost Galois coverings that looks different from Definition \ref{AlmostGaloisCov}. These definitions are equivalent to each other. Indeed, to show that \cite[Definition V.12.2]{AGT16} means Definition \ref{AlmostGaloisCov}, it is sufficient to apply \cite[Lemma V.12.4 and Lemma V.12.6]{AGT16}. To deduce the converse, it suffices to note that $f:A \to B$ is $I$-almost faithfully flat. This follows from \cite[Proposition 1.9.1 (1)]{An18} and the fact that an $I$-almost \'etale map is $I$-almost faithfully flat (see \cite[1.8.1. Remarques (1)]{An18}).

\begin{proposition}
\label{AlmostGaloisVanish}
Assume that a $V$-algebra map $A \to B$ is an $I$-almost $G$-Galois covering. Then $H^i(G,M)$ is $I$-almost zero for all $i>0$ and any $B$-module $M$ with a semi-linear $G$-action.
\end{proposition}

\begin{proof}
See \cite[Proposition V.12.8]{AGT16} for the proof.
\end{proof}

We need the following exact sequence of Milnor's type.

\begin{proposition}
\label{MilnorExact}
Let $G$ be a finite group and let $\{\cdots \to M_2 \to M_1\}$ be an inverse system of  $\mathbb{Z}[G]$-modules such that $\mathbf{R}^1\varprojlim_{i \in \mathbb{N}}M_i=0$. Then for any $n \in \mathbb{Z}$, there is a short exact sequence
$$
0 \to \mathbf{R}^1\varprojlim_{i \in \mathbb{N}}\Ext^{n-1}_{\mathbb{Z}[G]}(\mathbb{Z},M_i) \to \Ext^{n}_{\mathbb{Z}[G]}(\mathbb{Z},\varprojlim_{i \in \mathbb{N}}M_i) \to \varprojlim_{i \in \mathbb{N}}\Ext^{n}_{\mathbb{Z}[G]}(\mathbb{Z},M_i) \to 0.
$$
Moreover, the same conclusion holds in the category of almost $V$-modules with respect to a basic setup $(V,I)$.
\end{proposition}

\begin{proof}
Since ${\bf R}^1\varprojlim_{i \in \mathbb{N}} M_i=0$ by assumption, we get the short exact sequence $0 \to \varprojlim_{i \in \mathbb{N}} M_i \to \prod_{i=1}^\infty M_i \to \prod_{i=1}^\infty M_i \to 0$ in view of the definition of derived limits. This induces a long exact sequence
$$
\cdots \to \Ext_{\mathbb{Z}[G]}^n(\mathbb{Z},\varprojlim_{i \in \mathbb{N}} M_i) \to \prod_{i=1}^\infty \Ext_{\mathbb{Z}[G]}^n(\mathbb{Z},M_i) \to \prod_{i=1}^\infty \Ext_{\mathbb{Z}[G]}^n(\mathbb{Z},M_i)
$$
$$
\to \Ext_{\mathbb{Z}[G]}^{n+1}(\mathbb{Z},\varprojlim_{i \in \mathbb{N}} M_i) \to \cdots. 
$$
By the definition of derived limits, we obtain the desired exact sequence. The almost analogue is proved in the same manner, where $\{\cdots \to M_2 \to M_1\}$ is an inverse system of $V[G]$-modules under the assumption that  $\mathbf{R}^1\varprojlim_{i \in \mathbb{N}}M_i$ is $I$-almost zero.
\end{proof}

The tilt of commutative rings will be important for our study.

\begin{definition}
\label{tiltingring}
Let $A$ be a ring and let $p>0$ be a prime. Then we define the \textit{tilt} to be
$$
A^\flat:=\varinjlim \{\cdots \xrightarrow{F} A/pA \xrightarrow{F} A/pA\},
$$
where $F$ is the Frobenius endomorphism on $A/pA$. Let $\Phi^k_A:A^\flat \to A/pA$ be the projection map defined by $(\ldots,a_2,a_1) \in A^\flat \mapsto a_k \in A/pA$.
\end{definition}

We will exclusively consider the case when $A \ne pA$. Notice that $A^\flat$ is naturally a commutative ring containing $\mathbb{F}_p$. The map $\Phi^k_A$ is a ring map. Moreover, it is easy to check that $A^\flat$ is a perfect $\mathbb{F}_p$-algebra.

\begin{lemma}[Topological almost Nakayama's lemma]
\label{TopNakayama}
Let $\varpi \in V$ be a regular element with a system of $p$-power elements $\{\varpi^{1/p^n}\}_{n \in \mathbb{N}}$ in $V$ and thus, defining a basic setup. Let $M \to N$ be a map of $V$-modules. If $N$ is $\varpi$-adically separated and the induced map $M/\varpi M \to N/\varpi N$ is $(\varpi)^{1/p^\infty}$-almost surjective, then $M \to N$ is $(\varpi)^{1/p^\infty}$-almost surjective.
\end{lemma}

\begin{proof}
The proof given in \cite[Proposition 3.3]{IS22} works with minor modifications, so we omit the proof.
\end{proof}

\section{Main theorems}

Let $R=W(k)[[x_2,\ldots,x_d]]$ be a complete regular local ring with perfect residue field $k$ of characteristic $p>0$, where $W(k)$ is the ring of Witt vectors. Let $R^+$ be an absolute integral closure and define
$$
R_\infty:=\bigcup_{n>0}R[p^{1/p^n},x_2^{1/p^n},\ldots,x_d^{1/p^n}]
$$
as a subring of $R^+$. As shown in \cite[Proposition 4.9]{Shi16}, the $p$-adic completion $\widehat{R_\infty}$ is an integral perfectoid ring. Moreover, it is easy to see that $R \to R_\infty$ is faithfully flat. Then we apply \cite[Theorem 0.1]{Ye18} to conclude that
\begin{equation}
\label{BigMac12345}
\widehat{R_\infty}~\mbox{is an integral perfectoid balanced big Cohen-Macaulay}~R\mbox{-algebra}.
\end{equation}
We have integral extensions:
$$
R \hookrightarrow R_\infty \hookrightarrow R^+,
$$
where $R^+$ is the absolute integral closure of $R$. Recently, Bhatt \cite{Bh20} proved that $\widehat{R^+}$ is a balanced big Cohen-Macaulay $R$-algebra. By tilting, we have ring extensions:
\begin{equation}
\label{basicextension1}
k[[p^\flat,x_2^\flat,\ldots,x_d^\flat]] \hookrightarrow (R_\infty)^\flat \hookrightarrow (R^+)^\flat.
\end{equation}
Bhatt's result asserts that $p^\flat,x_2^\flat,\ldots,x_d^\flat$ is a regular sequence on $(R^+)^\flat$. On the other hand, Heitmann \cite{Hei22} proved that $(R^+)^\flat$ is an integral domain. However, we still do not much understand $(R^+)^\flat$ as an $(R_\infty)^\flat$-algebra. For instance, it is certainly not true that this is an integral extension. The reason is that even though $R_\infty$ is already non-Noetherian, the integral extension $R_\infty \to R^+$ is far from being finitely generated. We will see that $(R_\infty)^\flat$ is the $p^\flat$-adic completion of the perfect closure of $k[[p^\flat,x_2^\flat,\ldots,x_d^\flat]]$. Thus, $(R_\infty)^\flat$ is relatively easy to understand. Let us introduce another intermediate ring $R_\infty \to R_{\infty,p} \to R^+$ as follows.

\begin{definition}
\label{p-maximaletale}
Let $R_{\infty,p}$ be a unique maximal integral extension over $R_\infty$ such that $R_{\infty,p}$ is a subring of $R^+$ and the localization $R_\infty[1/p] \to R_{\infty,p}[1/p]$ is a filtered colimit of finite \'etale $R_\infty[1/p]$-algebras contained in $R^+[1/p]$. 
\end{definition}

The ring $R_{\infty,p}$ was originally introduced and studied in Faltings' groundbreaking work on $p$-adic Hodge theory (see \cite{AGT16} for connections with $p$-adic Simpson correspondence and \cite{Fa02} for $p$-adic Hodge theory). Intuitively, the extension $R_\infty \to R_{\infty,p}$ is large, but $R_{\infty,p} \to R^+$ is substantially large. The following result has long been known.

\begin{theorem}
\label{AlmostPure1}
Let the notation be as above. Then $R_{\infty,p}$ is a $(p)^\infty$-almost Cohen-Macaulay integrally closed domain in the sense that $p,x_2,\ldots,x_d$ is a $(p)^\infty$-almost regular sequence. Moreover, the $p$-adic completion of $R_{\infty,p}$ is a $p$-torsion free integral perfectoid ring.
\end{theorem}

\begin{proof}
We prove that the stated properties hold on $R_{\infty,p}$. Write $R_{\infty,p}[1/p]$ as a filtered direct system of finite \'etale $R_{\infty}[1/p]$-algebras $\{T'_\lambda\}_{\lambda \in \Lambda}$ such that $T'_\lambda \hookrightarrow R_{\infty,p}[1/p]$. Let $T_\lambda$ be the integral closure of $R_\infty$ in $T'_\lambda$. Then it follows that $R_\infty \to T_\lambda$ becomes finite \'etale after inverting $p$ and $R_{\infty,p}=\varinjlim_{\lambda \in \Lambda} T_\lambda$. Then by Witt-perfect almost purity theorem \cite[Theorem 5.9]{NS18}, $R_\infty \to T_\lambda$ is $(p)^\infty$-almost finite \'etale. Thus, we get that $p,x_2,\ldots,x_d$ forms a $(p)^\infty$-almost regular sequence on $T_\lambda$ because $R_\infty$ is a balanced big Cohen-Macaulay $R$-algebra. The same property remains to hold on $R_{\infty,p}=\varinjlim_{\lambda \in \Lambda} T_\lambda$. Finally, it remains to check that $R_{\infty,p}/\fm R_{\infty,p}$ is not $(p)^\infty$-almost zero. But this follows from \cite[Lemma 6.3]{NS23}.
\end{proof}

Now we have a tower of ring extensions which extends $(\ref{basicextension1})$: 
\begin{equation}
\label{basicextension2}
k[[p^\flat,x_2^\flat,\ldots,x_d^\flat]] \hookrightarrow (R_{\infty})^\flat \hookrightarrow (R_{\infty,p})^\flat \hookrightarrow (R^+)^\flat.
\end{equation}

Our main purpose in the rest of the article is to study the tower 
$(\ref{basicextension2})$, which is the main motivation of Problem \ref{tiltedRing}. Although our attempt to access $(R^+)^\flat$ via a certain intermediate ring has not yet been successful, it would be a good idea to access $(R_{\infty,p})^\flat$ in terms of the reasonably small ring $(R_\infty)^\flat$. Our first goal it to prove that $(R_\infty)^\flat$ is close to being Noetherian.

\begin{proposition}
\label{basictilting}
$(R_{\infty})^\flat$ is equal to the $p^\flat$-adic completion of the directed perfection of the power-series ring $k[[p^\flat,x_2^\flat,\ldots,x_d^\flat]]$.
\end{proposition}

\begin{proof}
We follow the notation as in \cite[Example 3.24]{INS23}. Let $R_n:=R[p^{1/p^n},x_2^{1/p^n},\ldots,x_d^{1/p^n}]$, so it yields a perfectoid tower $\{R_n\}_{n \ge 0}$ (see \cite{INS23} for this notion). For each integer $k>0$, we define the $k$-th small tilt:
$$
(R_k)^{s.\flat}_{(p)}:=\varprojlim\big\{\cdots \xrightarrow{F_{k+1}} R_{k+1}/pR_{k+1} \xrightarrow{F_k} R_k/pR_k\big\},
$$
where $F_k$ is induced by the Frobenius map $\Frob:R_{k+1}/pR_{k+1} \to R_{k+1}/pR_{k+1}$. That is, the Frobenius map factors as $R_{k+1}/pR_{k+1} \twoheadrightarrow R_k/pR_k \hookrightarrow  R_{k+1}/pR_{k+1}$. Then the colimit of the commutative diagram
$$
\begin{CD}
\vdots @. \vdots \\
@VF_{k+1}VV @VF_{k+2}VV \\
R_{k+1}/pR_{k+1} @>>> R_{k+2}/pR_{k+2} \\
@VF_kVV @VF_{k+1}VV \\
R_k/pR_k @>>> R_{k+1}/pR_{k+1} \\
\end{CD}
$$
induces a ring map
$$
(R_k)^{s.\flat}_{(p)} \cong k[[(p^\flat)^{1/p^k},\ldots,(x_d^\flat)^{1/p^k}]] \hookrightarrow 
(R_{k+1})^{s.\flat}_{(p)} \cong k[[(p^\flat)^{1/p^{k+1}},\ldots,(x_d^\flat)^{1/p^{k+1}}]].
$$
This is a purely inseparable extension of complete regular local rings. In other words, $(R_{k+1})^{s.\flat}_{(p)}=\big((R_{k})^{s.\flat}_{(p)}\big)^{1/p}$. Let us set
$$
(R_\infty)^{s.\flat}_{(p)}:=\bigcup_{k \ge 0} (R_k)^{s.\flat}_{(p)},
$$
which is the directed perfection of $k[[p^\flat,x_2^\flat,\ldots,x_d^\flat]]$. There is a natural ring map $(R_\infty)^{s.\flat}_{(p)} \to (R_\infty)^\flat$, which further extends to the $p^\flat$-adic completion
\begin{equation}
\label{completering}
\widehat{(R_\infty)^{s.\flat}_{(p)}} \to (R_\infty)^\flat,
\end{equation}
as the latter ring is already $p^\flat$-adically complete. Since both the source and the target of $(\ref{completering})$ modulo $(p^\flat)^k$ coincide with each other for any $k>0$, it follows from \cite[Theorem 8.4]{M86} that $\widehat{(R_\infty)^{s.\flat}_{(p)}} \cong (R_\infty)^\flat$, as desired.
\end{proof}

\begin{lemma}
\label{CompleteInt1}
Let $A \hookrightarrow T$ be an integral extension such that $A$ is a Noetherian domain and $T \cong \prod T_i$ is a finite product of integral domains $T_i$. Let $x \in A$ be a nonzero element that is regular on $T$. Then $T$ is integrally closed in $T[1/x]$ if and only if $T$ is completely integrally closed in $T[1/x]$.
\end{lemma}

\begin{proof}
Since $A$ is Noetherian, we can apply the same proof of \cite[Proposition 7.1]{NS23} to each direct factor of $T$.
\end{proof}

The following result had been obtained by Andreatta (see \cite[Proposition A.6]{An06}). Our proof is different from the one given by Andreatta, since we avoid the use of the precise calculation of ramification theory. Instead, we use big Cohen-Macaulay algebras of some type constructed by Gabber and Ramero in \cite{GR18}.

\begin{proposition}
\label{valuationtilting}
Let the notation be as above. Then the $p$-adic completion $\widehat{R_\infty}$ is a local integral domain. Moreover, $\widehat{R_\infty}$ is perfectoid and integrally closed in the field of fractions of $\widehat{R_\infty}$. Moreover, $(R_\infty)^\flat$ is also an integrally closed local domain.
\end{proposition}

\begin{proof}
Let $\widehat{R_\infty}$ be the $p$-adic completion of $R_\infty$. Then it readily follows that it is a perfectoid ring. Recall that $R_\infty$ has the unique maximal ideal $\fM$. As any maximal ideal of $\widehat{R_\infty}$ contains $p$, $\fM \widehat{R_\infty}$ is the unique maximal ideal, showing that $\widehat{R_\infty}$ is a local ring. Let us show that it is an integral domain. For this, note that $R_\infty \to R^+$ is an integral extension and $R_\infty$ is integrally closed in $R_\infty[1/p]$. Using this fact, one deduces that the natural map $R_\infty/p^nR_\infty \to R^+/p^nR^+$ is injective for all $n>0$. Taking the inverse limit, we get an injection $\widehat{R_\infty} \hookrightarrow \widehat{R^+}$. By the result of Heitmann \cite{Hei22}, we see that $\widehat{R^+}$ and thus $\widehat{R_\infty}$ are domains. Recall that $R_\infty$ is the ascending union of regular local rings $R_n=R[p^{1/p^n},x_2^{1/p^n},\ldots,x_d^{1/p^n}]$. Let $v_0:\Frac(R) \to \mathbb{Q} \cup \{\infty\}$ be the normalized $p$-adic valuation. In other words, for $0 \ne x \in R$, we set
$$
v_0(x):=\min\{k~|~k \in \mathbb{Z}~\mbox{such that}~x \in p^k R\}.
$$
That this is indeed a valuation follows from \cite[Theorem 6.7.8]{SwHu}. Since $R \to R_\infty$ is an integral extension of domains, $v_0$ naturally extends to a valuation on $v:\Frac(R_\infty) \to \mathbb{Q} \cup \{\infty\}$. For any nonzero $x \in \widehat{R_\infty}$, let us choose a Cauchy sequence $\{x_n\}$ such that $x_n \in R_\infty$ and $\displaystyle{\lim_{n \to \infty}} x_n=x$. Then in view of \cite[Propositon 9.1.16]{GR18}, we get a valuation $v_\infty:\Frac(\widehat{R_\infty}) \to \mathbb{Q} \cup \{\infty\}$ uniquely by letting $v_\infty(x):=\displaystyle \lim_{n \to \infty} v(x_n)$. In other words, $v$ (resp. $v_\infty$) is nothing other than the $p$-adic valuation on $R_\infty$ (resp. $\widehat{R_\infty}$). Set $P:=\bigcup_{n>0}p^{1/p^n}\widehat{R_\infty}$. Then this is a prime ideal of $\widehat{R_\infty}$ and the center of $v_\infty$ is exactly $P$. Recall that $\widehat{R_\infty}$ is an integral perfectoid balanced big Cohen-Macaulay $R$-algebra in view of $(\ref{BigMac12345})$. Then by \cite[Theorem 17.5.96]{GR18}, there is a ring map
$$
\widehat{R_\infty} \to B
$$
such that $B$ is a balanced big Cohen-Macaulay algebra over $R$ and $B$ is an absolutely integrally closed domain. Since $R_n$ is a regular local ring, it follows that $B$ is faithfully flat over $R_n$. Hence $B$ is faithfully flat over $R_\infty$. In conclusion, the extension $\widehat{R_\infty} \to B$ is $p$-completely faithfully flat. Now let $x=a/b \in \Frac(\widehat{R_\infty})$ be a nonzero element with $a,b \in \widehat{R_\infty}$ that is integral over $\widehat{R_\infty}$. But then $x$ is also contained in $\Frac(B)$ that is integral over $B$, we must get $x \in B$, because $B$ is integrally closed in $\Frac(B)$. Thus, we have $a \in bB$, or equivalently, the multplication map $a:B \to B/bB$ is the zero map, which implies that \begin{equation}
\label{zeromap1}
a:B/p^nB \to B/(b,p^n)B~\mbox{is the zero map}. 
\end{equation}
Without losing generality, we may assume that $b \in \widehat{R_\infty}$ is not a unit. Set $\epsilon:=v_\infty(b)$. Then we can write $b=p^\epsilon c$ for some $c \in \widehat{R_\infty} \setminus P$. 

First assume that $c \in (\widehat{R_\infty})^\times$. In this case, we have $x \in \widehat{R_\infty}$ if and only if $cx \in \widehat{R_\infty}$. So it suffices to prove the latter. Then since $cx \in \Frac(\widehat{R_\infty})$ is still integral over $\widehat{R_\infty}$ and $cx=a/p^\epsilon \in \widehat{R_\infty}[1/p]$, it follows from \cite[Lemma 3.3]{NS23} and integral closedness of $R_\infty$ in $R_\infty[1/p]$ that $cx \in \widehat{R_\infty}$.

Next assume that $c \in \widehat{R_\infty}$ is not a unit. Then $c$ is contained in the unique maximal ideal. Let $\alpha \in \widehat{R_\infty}$ be any element. Quite generally, we claim that
\begin{equation}
\label{zeromap2}
c+p\alpha,p^n~\mbox{forms a regular sequence on}~\widehat{R_\infty}~\mbox{for all}~n>0.
\end{equation}
First, we prove that $p^n,c$ is a regular sequence on $\widehat{R_\infty}$. It is clear that $p^n \in \widehat{R_\infty}$ is a regular element, so we prove that the multiplication map $c:\widehat{R_\infty}/p^n\widehat{R_\infty} \to \widehat{R_\infty}/p^n\widehat{R_\infty}$ is injective. Let us write $\overline{c} \in \widehat{R_\infty}/p^n\widehat{R_\infty}$ for the image of $c \in \widehat{R_\infty}$. Since $\widehat{R_\infty}/p^n\widehat{R_\infty}=\varinjlim_{i>0} R_i/p^nR_i$, there is an integer $k>0$ such that $\overline{c} \in R_k/p^nR_k$. Choose a lift $c' \in R_k$ of $\overline{c}$. Set $\fp:=R_k \cap P$. Then $c' \notin \fp$ and $\fp$ is the only associated prime ideal of $R_k/p^nR_k$, the map $c':R_k/p^nR_k \to R_k/p^nR_k$ is injective because of $c \in \widehat{R_\infty} \setminus P$. Since $R_k/p^nR_k \to R_\infty/p^nR_\infty \cong \widehat{R_\infty}/p^n\widehat{R_\infty}$ is faithfully flat for all $n>0$, the base change of $c':R_k/p^nR_k \to R_k/p^nR_k$ gives the injectivity of $c:\widehat{R_\infty}/p^n\widehat{R_\infty} \to \widehat{R_\infty}/p^n\widehat{R_\infty}$. Thus, we proved that the sequence $p^n,c$ is regular on $\widehat{R_\infty}$. By letting $n=1$, the sequence $p,c+p\alpha$ is also regular on $\widehat{R_\infty}$. Then it follows that $p^n,c+p\alpha$ is regular as well. Since $\widehat{R_\infty}$ is evidently $p$-adically complete, the assertion $(\ref{zeromap2})$ follows readily. We use this claim by taking $\alpha=0$.

Notice that $R_\infty/p^nR_\infty \cong \widehat{R_\infty}/p^n\widehat{R_\infty} \to B/p^nB$ is faithfully flat for all $n>0$. So the faithfully flat descent together with $(\ref{zeromap1})$ imply that the composite map
$$
a:\widehat{R_\infty}/p^n\widehat{R_\infty} \to \widehat{R_\infty}/(b,p^n)\widehat{R_\infty} \twoheadrightarrow \widehat{R_\infty}/(c,p^n)\widehat{R_\infty}
$$
must be the zero map for all $n>0$. Taking the inverse limit over $n \in \mathbb{N}$, we get that
\begin{equation}
\label{zeromap3}
a:\widehat{R_\infty} \to \widehat{\widehat{R_\infty}/c\widehat{R_\infty}}
\end{equation}
is the zero map, where $\widehat{\widehat{R_\infty}/c\widehat{R_\infty}}$ is the $p$-adic completion of $\widehat{R_\infty}/c\widehat{R_\infty}$.\footnote{As $\widehat{R_\infty}$ is not Noetherian, one cannot immediately conclude that its quotient ring is separated in the $p$-adic topology.} Now we have the commutative diagram:
$$
\begin{CD}
0 @>>> \widehat{R_\infty} @>\times c >> \widehat{R_\infty} @>>> \widehat{R_\infty}/c\widehat{R_\infty} @>>> 0 \\
@. @V\times p^nVV @V\times p^nVV @V\times p^nVV \\
0 @>>> \widehat{R_\infty} @>\times c >> \widehat{R_\infty} @>>> \widehat{R_\infty}/c\widehat{R_\infty} @>>> 0 \\
\end{CD}
$$
Thanks to $(\ref{zeromap2})$, the snake lemma yields a short exact sequence:
\begin{equation}
\label{zeromap4}
0 \to \widehat{R_\infty}/p^n \widehat{R_\infty} \xrightarrow{\times c} \widehat{R_\infty}/p^n \widehat{R_\infty} \to \widehat{R_\infty}/(c,p^n)\widehat{R_\infty} \to 0.
\end{equation}
Taking inverse limit of $(\ref{zeromap4})$ with respect to $n \in \mathbb{N}$, we obtain an isomorphism $\widehat{R_\infty}/c\widehat{R_\infty} \cong \widehat{\widehat{R_\infty}/c\widehat{R_\infty}}$. Now combining this fact with $(\ref{zeromap3})$, we find that $a:\widehat{R_\infty} \to \widehat{R_\infty}/c\widehat{R_\infty}$ is the zero map. Hence we have $a=cd$ for some $d \in \widehat{R_\infty}$ and 
$$
x=\frac{a}{b}=\frac{a}{p^\epsilon c}=\frac{cd}{p^\epsilon c}=\frac{d}{p^\epsilon} \in \widehat{R_\infty}[\frac{1}{p}] \subseteq  \Frac(\widehat{R_\infty}).
$$
By \cite[Lemma 3.3]{NS23}, we must get $x \in \widehat{R_\infty}$. 

As for $(R_\infty)^{\flat}$, we use Proposition \ref{basictilting} and note that $(R_\infty)^{\flat}$ is the $p^\flat$-adic completion of the perfect closure of the complete regular local ring $A:=k[[p^\flat,x_2^\flat,\ldots,x_d^\flat]]$. So by setting $v_0:\Frac(A) \to \mathbb{Q} \cup \{\infty\}$ to be
$$
v_0(x):=\min\{k~|~k \in \mathbb{Z}~\mbox{such that}~x \in (p^\flat)^k A\},
$$
the same proof in the case of $\widehat{R_\infty}$ goes through without making essential changes. This finishes the proof of the proposition.
\end{proof}

\begin{remark}
The $p$-completed ring $\widehat{R_\infty}$ plays a fundamental role in the computations of certain numerical invariants, such as $F$-signature or Hilbert-Kunz multiplicities and their mixed characteristic variants. See the papers \cite{BMPSTWW24} and \cite{CLMST22} for these topics. Recently, the paper \cite{HeiMa25} establishes that $\widehat{R_\infty}$ is completely integrally closed in its field of fractions by using a different method.
\end{remark}

We use Galois theory of commutative rings to establish the main result. The next result on the tilting correspondence for decompleted perfectoid rings is probably known to experts. As we are unable to find an appropriate reference, we give a proof. Following the convention in \cite{NS18}, let us introduce the categories: ${\bf{F.Et}}^a_{\rm{G}}(A)$ and ${\bf{F.Et}}^a_{\rm{G}}(A^\flat)$ for the basic setup $(A,(\varpi)^{1/p^\infty})$. The category ${\bf{F.Et}}^a_{\rm{G}}(A)$ (resp. ${\bf{F.Et}}^a_{\rm{G}}(A^\flat)$) is obtained as the localization of the category of $(\varpi)^{1/p^\infty}$-almost Galois coverings over $A$ (resp. $A^\flat$) with respect to the class of $(\varpi)^{1/p^\infty}$-almost isomorphisms (resp. $(\varpi^\flat)^{1/p^\infty}$-almost isomorphisms). For Henselization of rings along an ideal and their basic properties, we refer the reader to \cite[Tag 09XE]{Stacks}.

\begin{proposition}[Galois-tilting correspondence]
\label{GaloisCovProp}
Let $A$ be a $V$-algebra such that $V$ is a Witt-perfect valuation ring and there is a nonzero element $\varpi \in V$ such that $\varpi=p v$ for a unit $v \in V^\times$ and $\varpi^{1/p^n} \in V$ for all $n>0$. Assume that the Frobenius map $F:A/\varpi A \to A/\varpi A$ is surjective and $A$ is integrally closed in $A[1/\varpi]$. Assume further that $A$ is $\varpi$-adically Henselian and $\Spec(A)$ is connected. Then the tilting operation gives an equivalence of categories:
$$
{\bf{F.Et}}^a_{\rm{G}}(A) \xrightarrow{\Phi} {\bf{F.Et}}^a_{\rm{G}}(A^\flat).
$$
Under this correspondence, if $G$ is a finite group, a $(\varpi)^{1/p^\infty}$-almost $G$-Galois covering $B$ over $A$ goes to a $(\varpi^\flat)^{1/p^\infty}$-almost $G$-Galois covering $B^\flat$ over $A^\flat$.
\end{proposition}

\begin{proof}
By integral closedness of $A$ in $A[1/\varpi]$, it can be proved that the Frobenius map on $A/\varpi A$ induces an isomorphism $A/\varpi^{1/p}A \cong A/\varpi A$. First, we prove the proposition under the assumption that $A$ is $\varpi$-adically complete, which we will do so in what follows. By the classical almost purity theorem proved by Scholze, Kedlaya-Liu (see \cite[Corollary 5.12]{NS18} for the statement and \cite[Theorem 3.6.21]{KL15} for the tilting correspondence), the functor $\Phi$ gives an equivalence between the category of $(\varpi)^{1/p^\infty}$-almost finite \'etale extensions over $A$ and the category of $(\varpi^\flat)^{1/p^\infty}$-almost finite \'etale extensions over $A^\flat$. Its quasi-inverse is given by
\begin{equation}
\label{quasi-inverse123}
D \mapsto D^\sharp:=W(D)/(pu-[\varpi^\flat])W(D)~\mbox{for some unit}~u \in W(A^\flat)^\times,
\end{equation}
and we refer the reader to \cite[Proposition 2.1.9]{CS24} in the context of perfectoid rings. The functor $\Phi$ in the proposition is just the restriction to the full subcategory consisting of all $(\varpi)^{1/p^\infty}$-almost Galois coverings over $A$. We claim that the functor $\Phi$ lands in the stated category. Namely, pick a $(\varpi)^{1/p^\infty}$-almost $G$-Galois covering $A \to B$. Then we need to show that $A^\flat \to B^\flat$ is a $(\varpi^\flat)^{1/p^\infty}$-almost $G$-Galois covering. Since $A$ is assumed to be integrally closed in $A[1/\varpi]$, $A \to B$ is $(\varpi)^{1/p^\infty}$-almost finite \'etale, and we are working in the category of almost modules, we may replace $B$ by the integral closure of $B$ in $B[1/\varpi]$. In this situation, the Frobenius map on $B/\varpi B$ is surjective and we use this fact below. First of all, note that $B^\flat$ has an action by $G$ induced from the action on $B$ by the definition of tilts. Indeed, since the action of $G$ on $B$ restricts to the trivial action on $A$, we see that the action of $G$ on $B^\flat$ restricts to the trivial action on $A^\flat$. In particular, we obtain the natural map $A^\flat \to (B^{\flat})^G$. Consider the exact sequence: $0 \to B \xrightarrow{\times \varpi} B \to B/\varpi B \to 0$. It induces a long exact sequence
$$
0 \to B^G \xrightarrow{\times \varpi} B^G \to (B/\varpi B)^G \to H^1(G,B) \to \cdots.
$$
Since $H^1(G,B)$ almost vanishes in view of Proposition \ref{AlmostGaloisVanish} and $B^G$ is $(\varpi)^{1/p^\infty}$-almost isomorphic to $A$, we get
\begin{equation}
\label{almostiso10}
(B/\varpi B)^G~\mbox{is}~(\varpi)^{1/p^\infty}\mbox{-almost isomorphic to}~A/\varpi A.
\end{equation}
Again by Proposition \ref{AlmostGaloisVanish}, the $A$-module $H^1(G,B/\varpi B)$ is $(\varpi)^{1/p^\infty}$-almost zero. We apply Proposition \ref{MilnorExact} by letting $n=1$ to the inverse system of $\mathbb{Z}[G]$-modules:
$$
\{\cdots \xrightarrow{F} B/\varpi B \xrightarrow{F} B/\varpi B\},
$$
which we view as a sequence of $A^\flat$-modules via $A^\flat \to A^\flat/\varpi^\flat A \cong A/\varpi A$. Since the Frobenius $F$ is surjective on $B/\varpi B$, we have $\mathbf{R}^1\varprojlim B/\varpi B \cong 0$. Note that the Frobenius map on $B/\varpi B$ commutes with any ring map and the natural injection $\varpi^{1/p}B/\varpi B \to B/\varpi B$ is a map of $\mathbb{Z}[G]$-modules. So there is a short exact sequence of $\mathbb{Z}[G]$-modules $0 \to \varpi^{1/p}B/\varpi B \to B/\varpi B \xrightarrow{F} B/\varpi B \to 0$. Then we get an exact sequence
$$
0 \to \Hom_{\mathbb{Z}[G]}(\mathbb{Z},\varpi^{1/p}B/\varpi B) \to \Hom_{\mathbb{Z}[G]}(\mathbb{Z},B/\varpi B) \xrightarrow{F} \Hom_{\mathbb{Z}[G]}(\mathbb{Z},B/\varpi B)
$$
$$
\to \Ext^1_{\mathbb{Z}[G]}(\mathbb{Z},B/\varpi B) \to \cdots.
$$
Since the $B$-module $\varpi^{1/p}B/\varpi B$ has a semi-linear $G$-action, it follows from Proposition \ref{AlmostGaloisVanish} that $\Ext^1_{\mathbb{Z}[G]}(\mathbb{Z},\varpi^{1/p}B/\varpi B)=H^1(G,\varpi^{1/p}B/\varpi B)$ is $(\varpi)^{1/p^\infty}$-almost zero. Thus, the inverse system $\{\Hom_{\mathbb{Z}[G]}(\mathbb{Z},B/\varpi B)\}$, which is defined by the Frobenius map on $B/\varpi B$, satisfies the $(\varpi)^{1/p^\infty}$-almost Mittag-Leffler condition. So we have $\mathbf{R}^1 \varprojlim \Hom_{\mathbb{Z}[G]}(\mathbb{Z},B/\varpi B)$ is $(\varpi^\flat)^{1/p^\infty}$-almost zero. Hence by Proposition \ref{MilnorExact}, $H^1(G,B^\flat)$ is $(\varpi^\flat)^{1/p^\infty}$-almost isomorphic to $\varprojlim H^1(G,B/\varpi B)$. It follows that $H^1(G,B^\flat)$ is $(\varpi^\flat)^{1/p^\infty}$-almost zero. By taking cohomology of the exact sequence $0 \to B^\flat \xrightarrow{\times \varpi^\flat} B^\flat \to B^\flat/\varpi^\flat B^\flat \to 0$, we find that $(B^\flat)^G/\varpi^\flat (B^\flat)^G$ is almost isomorphic to $(B^\flat/\varpi^\flat B^\flat)^G$. Since $A/\varpi A \cong A^\flat/\varpi^\flat A^\flat$ and $B/\varpi B \cong B^\flat/\varpi^\flat B^\flat$, it follows from $(\ref{almostiso10})$ that
$$
A^\flat/\varpi^\flat A^\flat~\mbox{is}~(\varpi^\flat)^{1/p^\infty}\mbox{-almost isomorphic to}~(B^\flat)^G/\varpi^\flat (B^\flat)^G.
$$
Note that $(B^\flat)^G$ is $\varpi^\flat$-adically separated as it is a subring of the $\varpi^\flat$-adically complete ring $B^\flat$. Now we can apply almost Nakayama's lemma (see Lemma \ref{TopNakayama}) to get that $(B^\flat)^G$ is $(\varpi^\flat)^{1/p^\infty}$-almost isomorphic to $A^\flat$, which checks the first condition of Definition \ref{AlmostGaloisCov}. The verification of the second condition of Definition \ref{AlmostGaloisCov} is left as an exercise, as it can be done by taking the quotient modulo $\varpi$ (or $\varpi^\flat$) and utlizing \cite[Theorem 5.3.27]{GR03}.

Next we claim that the quasi-inverse of $\Phi$ is given by $D \mapsto D^\sharp:=W(D)/(pu-[\varpi^\flat])W(D)$ for a $(\varpi^\flat)^{1/p^\infty}$-almost $G$-Galois extension $D$ over $A^\flat$. What needs to be proved is that $D^\sharp$ is a $(\varpi)^{1/p^\infty}$-almost $G$-Galois extension over $A$. The group $G$ acts on $W(D)$ compatibly with the reduction $W(D) \to D$ because any ring map on $D$ lifts to a unique ring map on $W(D)$. It is true that the element $pu-[\varpi^\flat]$ is fixed by $G$ because we have $u \in W(A^\flat)^\times$ by $(\ref{quasi-inverse123})$. Hence $G$ acts on $D^\sharp$ as $A$-automorphisms. For $n \in \mathbb{N}$, $W_n(A^\flat) \to W_n(D)$ is readily shown to be a $[\varpi^\flat]^{1/p^\infty}$-almost $G$-Galois extension.  In particular, $\Hom_{\mathbb{Z}[G]}(\mathbb{Z},W_n(D)) \cong W_n(D)^G$ is $[\varpi^\flat]^{1/p^\infty}$-almost isomorphic to $W_n(A)$. Hence the inverse system $\{\Hom_{\mathbb{Z}[G]}(\mathbb{Z},W_n(D))\}_{n \in \mathbb{N}}$  satisfies the $(\varpi)^{1/p^\infty}$-almost Mittag-Leffler condition. Thus, $\mathbf{R}^1 \varprojlim_{n \in \mathbb{N}}\Hom_{\mathbb{Z}[G]}(\mathbb{Z},W_n(D))$  is $(\varpi)^{1/p^\infty}$-almost zero. By Proposition \ref{MilnorExact}, 
$H^1(G,W(D))$ is $[\varpi^\flat]^{1/p^\infty}$-almost isomorphic to $\varprojlim_{n>0} H^1(G,W_n(D))$. Since Proposition \ref{AlmostGaloisVanish} gives that $H^1(G,W_n(D))$ is $[\varpi^\flat]^{1/p^\infty}$-almost zero, we have that $H^1(G,W(D))$ is $[\varpi^\flat]^{1/p^\infty}$-almost zero. We use this fact below. By taking the Galois cohomology of the short exact sequence
$$
0 \to W(D) \xrightarrow{\times (pu-[\varpi^\flat])} W(D) \to D^\sharp \to 0,
$$
we obtain that $W(D)^G/(pu-[\varpi^\flat])W(D)^G$ is $(\varpi)^{1/p^\infty}$-almost isomorphic to $D^{\sharp G}$. Since the functor of Witt vectors commutes with inverse limit and $A^\flat \to D^G$ is a $(\varpi^\flat)^{1/p^\infty}$-almost isomorphism, it follows that $W(A^\flat)$ is $[\varpi^\flat]^{1/p^\infty}$-almost isomorphic to $W(D)^G$. Furthermore, since $W(A^\flat)/(pu-[\varpi^\flat])W(A^\flat) \cong A$, we conclude that $D^{\sharp G}$ is $(\varpi)^{1/p^\infty}$-almost isomorphic to $A$. This verifies the first condition of Definition \ref{AlmostGaloisCov}. The verification of the second condition is again left as an exercise.

Finally, let us consider the general case and denote by $\widehat{A}$ the $\varpi$-adic completion of $A$. Recall the functor $\Phi_2:{\bf{F.Et}}^a(A) \to {\bf{F.Et}}(A[1/\varpi])$ in \cite[Corollary 5.12]{NS18}. Then $\Phi_2$ obviously sends a $(\varpi)^{1/p^\infty}$-almost $G$-Galois covering over $A$ to a $G$-Galois covering over $A[1/\varpi]$. Conversely, let $C \in {\bf{F.Et}}_{\rm{G}}(A[1/\varpi])$ with Galois group $G$. Then there is a $(\varpi)^{1/p^\infty}$-almost finite \'etale covering $B$ over $A$ such that $C \cong B[1/\varpi]=\Phi_2(B)$. That the functor $\Phi_2$ is an equivalence yields that the restricted action by $G$ on $B$ endows $A \to B$ with the structure of a $(\varpi)^{1/p^\infty}$-almost $G$-Galois covering over $A$. Hence the functor $\Phi_2$ gives an equivalence of categories between $(\varpi)^{1/p^\infty}$-almost $G$-Galois coverings over $A$ and $G$-Galois coverings over $A[1/\varpi]$. In the same way, one finds that the functor $\Phi_4$ in \cite[Corollary 5.12]{NS18} has the same conclusion. Next, let us show that $\Phi_3$ gives an equivalence of categories between $G$-Galois coverings over $A[1/\varpi]$ and $G$-Galois coverings over $\widehat{A}[1/\varpi]$. To this aim, notice that for a commutative ring $S$, a $G$-Galois covering over $S$ is a $G$-torsor over $\Spec(S)$ by \cite[Proposition 5.3.16]{Sz09}. Under the present hypothesis including the one that $A$ is $\varpi$-adically Henselian, the desired claim follows from \cite[Corollary 2.1.22(c)]{BC22} (which generalizes \cite[Theorem 5.8.14]{GR03}). Hence we conclude that the functor $\Phi_1$ from \cite[Corollary 5.12]{NS18} gives an equivalence of categories between $(\varpi)^{1/p^\infty}$-almost $G$-Galois coverings over $A$ and $(\varpi)^{1/p^\infty}$-almost $G$-Galois coverings over $\widehat{A}$. Combining the case of the complete case established above, we conclude that $\Phi$ is an equivalence.
\end{proof}

We are ready to prove the main result. We establish it for $(R_{\infty,p})^\flat$, because the proof of the case of $\widehat{R_{\infty,p}}$ can be derived in the same manner with less complications.

\begin{theorem}
\label{Galoistilting}
Let the notation be as above. Then we have the following assertions.
\begin{enumerate}
\item
$(R_{\infty,p})^\flat$ is a $p^\flat$-adically complete perfect integral domain, and $\Phi^k_{R_{\infty,p}}:(R_{\infty,p})^\flat \to R_{\infty,p}/pR_{\infty,p}$ is surjective and the kernel of $\Phi^k_{R_{\infty,p}}$ is the principal ideal generated $(p^\flat)^k$ for any $k \ge 1$ (see Definition \ref{tiltingring}).

\item
In view of $(1)$, define $B$ to be the unique maximal integral extension over $(R_\infty)^\flat$ such that $B$ is an integrally closed domain which is a subring of the integral domain $(R_{\infty,p})^\flat$ and the localization map $(R_\infty)^\flat[1/p^\flat] \to B[1/p^\flat]$ is a filtered colimit of all finite \'etale $(R_\infty)^\flat[1/p^\flat]$-algebras contained in $(R_{\infty,p})^\flat[1/p^\flat]$ (cf. Definition \ref{p-maximaletale}). Then $B$ is a perfect ring that is integral over $(R_{\infty})^\flat$ and $(R_{\infty,p})^\flat$ is equal to the $p^\flat$-adic completion of $B$.

\item
Set $A:=k[[p^\flat,x_2^\flat,\ldots,x_d^\flat]]$ and let $A^{1/p^\infty}$ be the perfect closure of $A$. Define $C$ to be the unique maximal integral extension over $A^{1/p^\infty}$ such that $C$ is an integrally closed domain which is a subring of the absolute integrally closed domain $A^+$ and the localization map $A^{1/p^\infty}[1/p^\flat] \to C[1/p^\flat]$ is a filtered colimit of all finite \'etale $A^{1/p^\infty}[1/p^\flat]$-algebras contained in $A^+[1/p^\flat]$. Then $\widehat{C} \cong (R_{\infty,p})^\flat$, where $\widehat{C}$ is the $p^\flat$-adic completion of $C$.
\end{enumerate}
\end{theorem}

\begin{proof}
$(1)$: By Theorem \ref{AlmostPure1}, the Frobenius endomorphism $R_{\infty,p}/pR_{\infty,p} \to R_{\infty,p}/pR_{\infty,p}$ is surjective with kernel generated by $p^{1/p}$. Using this fact, it can be checked that $(R_{\infty,p})^\flat$ is $p^\flat$-adically complete by applying the proof of \cite[Theorem 3.10(1)(a)]{INS23}, and $\Phi^k_{R_{\infty,p}}$ is surjective with kernel generated by $(p^\flat)^k$ for $k \ge 1$. It remains to prove that $(R_{\infty,p})^\flat$ is a domain. For this, recall from \cite{Hei22} that the $p$-adic completion of the absolute integral closure $R^+$ is a domain. This in turn implies that $(R^+)^\flat$ is a domain in view of \cite[Proposition 2.6]{EHS23}. Let us show that the natural ring injection $R_{\infty,p} \hookrightarrow R^+$ induces an injection $R_{\infty,p}/pR_{\infty,p} \hookrightarrow R^+/pR^+$. For this, it suffices to see that $p R_{\infty,p}=R_{\infty,p} \cap pR^+$. This easily follows from the fact that $R_{\infty,p}$ is integrally closed in $R_{\infty,p}[1/p]$. Applying the tilting operation, we see that $R_{\infty,p}/pR_{\infty,p} \hookrightarrow R^+/pR^+$ lifts to an injection $(R_{\infty,p})^\flat \hookrightarrow (R^+)^\flat$. Thus, $(R_{\infty,p})^\flat$ is an integral domain.

$(2)$: By the definition of $R_{\infty,p}$, we can write
\begin{equation}
\label{colimitetale}
R_{\infty,p}=\varinjlim S,
\end{equation}
where $R_\infty \to S$ runs over all integral extensions of normal domains such that $R_\infty[1/p] \to S[1/p]$ is finite \'etale. In particular, $\Frac(R_\infty) \to \Frac(S)$ is a finite field extension. Then there exists an integral extension $R_\infty[1/p] \to T'$ which factors as $R_\infty[1/p] \to S[1/p] \to T'$ such that $R_\infty[1/p] \to T'$ is a $G$-Galois covering by applying \cite[Lemma 9.4]{NS23}. In particular, $R_\infty[1/p] \to T'=T[1/p]$ is finite \'etale. Let $T$ be the integral closure of $R_\infty$ in $T'$. In this case, $T$ (and also $T'$) is a normal domains. Indeed, let $L$ be the Galois closure of $\Frac(R_\infty[1/p]) \to \Frac(T[1/p])$. Then we can show that $T'$ can be taken to be the integral closure of $R_\infty[1/p]$ in $L$ because $R_\infty[1/p]$ is a normal domain.

To prove the theorem, after replacing $S$ with $T$, we may assume that $R_\infty[1/p] \to S[1/p]$ is a $G$-Galois covering (in particular, finite \'etale). Then $R_\infty \to S$ is a $(p)^{1/p^\infty}$-almost finite \'etale extension by the Witt-perfect almost purity theorem (see \cite[Theorem 5.9]{NS18}). Now we want to show that $R_\infty \to S$ is a $(p)^{1/p^\infty}$-almost $G$-Galois covering. First, we have $R_\infty=S^G$ by the normality of $R_\infty$ and $R_\infty[1/p]=S[1/p]^G$, which verifies the first condition of Definition  \ref{AlmostGaloisCov}. As for the second condition, note that the diagonal mappping $S \to \prod_{g \in G} S$ factors as
$$
S \xrightarrow{h_1} S \otimes_{R_\infty} S \xrightarrow{h_2} \prod_{g \in G}S,
$$
where $h_1(b):=1 \otimes b$ and $h_2(b_1 \otimes b_2):=(g(b_1)b_2)_{g \in G}$ as defined in Definition  \ref{AlmostGaloisCov}. As $R_\infty[1/p] \to S[1/p]$ is a $G$-Galois covering, we get that $h_2[1/p]$ is bijective. Moreover, as the diagonal mapping $S \to \prod_{g \in G} S$ is integral, it follows that $h_2$ is also integral. Since $R_\infty \to S$ is $(p)^{1/p^\infty}$-almost finite, we find that $h_2:S \to S \otimes_{R_\infty} S$ is $(p)^{1/p^\infty}$-almost finite \'etale by base change invariance of almost finite \'etaleness. In particular, $p$ is a $(p)^{1/p^\infty}$-almost regular element on $S \otimes_{R_\infty} S$. As $h_2[1/p]$ is bijective, se see that $h_2$ is $(p)^{1/p^\infty}$-almost injective. We want to show that $h_2$ is $(p)^{1/p^\infty}$-almost surjective. Consider the composite map $S \xrightarrow{h_1} S \otimes_{R_\infty} S \to h_2(S \otimes_{R_\infty}S)$. Since $h_2$ is $(p)^{1/p^\infty}$-almost injective, it follows that $S \to h_2(S \otimes_{R_\infty}S)$ is $(p)^{1/p^\infty}$-almost finite \'etale. Now we have an injection $h_2(S \otimes_{R_\infty}S) \hookrightarrow \prod_{g \in G}S$. We have already seen that this map is integral and becomes bijective after inverting $p$. Moreover, it is evident that $\prod_{g \in G} S$ is integrally closed after inverting $p$. On the other hand, the $(p)^{1/p^\infty}$-almost finite \'etaleness of $S \to h_2(S \otimes_{R_\infty}S)$ gives that the integral closure of $h_2(S \otimes_{R_\infty}S)$ in $h_2(S \otimes_{R_\infty}S)[1/p]$ is $(p)^{1/p^\infty}$-almost isomorphic to $h_2(S \otimes_{R_\infty}S)$. These facts combine together to conclude that $h_2(S \otimes_{R_\infty}S) \hookrightarrow \prod_{g \in G}S$ is $(p)^{1/p^\infty}$-almost surjective. That is, $h_2$ is $(p)^{1/p^\infty}$-almost bijective and $R_\infty \to S$ is a $(p)^{1/p^\infty}$-almost $G$-Galois covering, as desired.
 
Since $R_n$ is $\varpi$-adically complete, it is $\varpi$-adically Henselian by \cite[Tag 0ALJ]{Stacks}. Then the filtered colimit $R_\infty$ is also $\varpi$-Henselian by \cite[Tag 0FWT]{Stacks}. Moreover, $\Spec(R_\infty)$ is obviously connected. Then it follows that $(R_\infty)^\flat \to S^\flat$ is a $(p^\flat)^{1/p^\infty}$-almost $G$-Galois covering by Proposition \ref{GaloisCovProp}. The natural induced map
\begin{equation}
\label{almostsurjective1}
(R_\infty)^\flat \to (S^\flat)^G
\end{equation}
is injective and $(p^\flat)^{1/p^\infty}$-almost surjective. We claim that $(\ref{almostsurjective1})$ is indeed an equality, which will imply that $(R_\infty)^\flat \to S^\flat$ is an integral extension. So let us prove this claim. By definition, $R_\infty$ is integrally closed in $R_\infty[1/p]$. So $R_\infty$ is completely integrally closed in $R_\infty[1/p]$ by Lemma \ref{CompleteInt1}. Hence it follows from \cite[Main Theorem 1]{EHS23} that $(R_\infty)^\flat$ is completely integrally closed in $(R_\infty)^\flat[1/p^\flat]$. Pick an element $x \in (S^\flat)^G$. By $(p^\flat)^{1/p^\infty}$-almost surjectivity of $(\ref{almostsurjective1})$, we obtain $x \in (R_\infty)^\flat[1/p^\flat]$ and 
$$
(p^\flat)^{1/p^n} x \in (R_\infty)^\flat;~\forall n>0.
$$
In other words, we have $p^\flat x^{p^n} \in (R_\infty)^\flat$ for all $n>0$, which says that $x \in (R_\infty)^\flat[1/p^\flat]$ is almost integral over $(R_\infty)^\flat$. So we have $x \in (R_\infty)^\flat$, as desired. Next set
\begin{equation}
\label{colimitetale2}
\widetilde{B}:=\varinjlim S^\flat,
\end{equation}
where the direct limit runs over all $S$ appearing in the directed system of $(\ref{colimitetale})$. Then $(R_{\infty})^\flat \hookrightarrow \widetilde{B} \hookrightarrow (R^+)^\flat$ and we have just proved that $(R_\infty)^\flat \to \widetilde{B}$ is an integral extension. Since the $p$-adic completion of $S$ is integral perfectoid, it follows that $\Phi^k_{S}:S^\flat \to S/pS$ is surjective with kernel $p^\flat S^\flat$. This implies that $\widetilde{B}$ surjects onto $R_{\infty,p}/pR_{\infty,p}$ because of $(\ref{colimitetale})$. Hence $\widetilde{B}/p^\flat \widetilde{B} \cong R_{\infty,p}/pR_{\infty,p}$ follows. Let $\widehat{\widetilde{B}}$ be the $p^\flat$-adic completion. By topological Nakayama's lemma \cite[Theorem 8.4]{M86}, we get an isomorphism $\widehat{\widetilde{B}} \cong (R_{\infty,p})^\flat$, as desired.

Now it remains to prove that $\widetilde{B}$ defined in $(\ref{colimitetale2})$ is identified with $B$ as asserted in $(2)$ of the theorem. By construction, it is clear that $\widetilde{B} \subseteq B$. It suffices to show the following assertion.
\begin{enumerate}
\item[$\bullet$]
Let $T'$ be an integral domain such that $(R_\infty)^\flat[1/p^\flat]  \hookrightarrow T' \hookrightarrow (R_{\infty,p})^\flat[1/p^\flat]$ and $(R_\infty)^\flat[1/p^\flat] \to T'$ is finite \'etale. Denote by  $T$ the integral closure of $(R_\infty)^\flat$ in $T'$. Then $T \subseteq \widetilde{B}$.
\end{enumerate}
Let $T^\sharp$ be the untilt of $T$. Then $T^\sharp$ is a $(p)^{1/p^\infty}$-almost finite \'etale extension over $\widehat{R_\infty}$ and $T^\sharp \subseteq \widehat{R_{\infty,p}}$. By applying \cite[Corollary 5.12]{NS18}, we find that $T^\sharp$ arises as a $p$-adic completion of $E$ which is a $(p)^{1/p^\infty}$-almost finite \'etale extension over $R_\infty$ such that $R_\infty \subseteq E \subseteq R_{\infty,p}$. Since $E^\flat=T$ by tilting correspondence for perfectoid rings, it follows that $T \subseteq \widetilde{B}$ in view of ($\ref{colimitetale2})$ and we are done.

$(3)$: By Proposition \ref{basictilting}, we know that the $p^\flat$-adic completion $\widehat{A^{1/p^\infty}}$ is equal to $(R_\infty)^\flat$. Moreover, $(R_0)_{(p)}^{s.\flat}=k[[p^\flat,x_2^\flat,\ldots,x_d^\flat]]$ is $p^\flat$-adically Henselian. So the colimit $A^{1/p^\infty}=\varinjlim_{n>0} A^{1/p^n}$ is also $p^\flat$-adically Henselian in view of \cite[Tag 0FWT]{Stacks}. By taking these observations into account and considering two basic setups $(A^{1/p^\infty},(p^\flat)^{1/p^\infty})$ and $((R_\infty)^\flat,(p^\flat)^{1/p^\infty})$, the claim follows from the \'etale-correspondence \cite[Corollary 5.12]{NS18} as follows. Let $A^{1/p^\infty} \to T$ be a $(p^\flat)^{1/p^\infty}$-almost finite \'etale extension such that $T$ is a normal domain that is a subring of $A^+$. Then, it suffices to see that the functor $\Phi_1$ of \cite[Corollary 5.12]{NS18} sends $T$ to a normal domain that is a subring of $(R_{\infty,p})^\flat$. The functor $\Phi_1$ is given by $T \mapsto \widehat{T}:=(R_\infty)^\flat \otimes_{A^{1/p^\infty}} T$. Then $(R_\infty)^\flat \to \widehat{T}$ is $(p^\flat)^{1/p^\infty}$-almost finite \'etale (or equivalently, $(R_\infty)^\flat[1/p^\flat] \to \widehat{T}[1/p^\flat]$ is finite \'etale). By Proposition \ref{valuationtilting}, $(R_\infty)^\flat$ is a $p^\flat$-adically complete integrally closed domain. In this situation, we can apply \cite[Proposition 4.1]{NS18} to say that $\widehat{T}$ is $p^\flat$-adically complete. Since $(R_\infty)^\flat$ is the $p^\flat$-adic completion of $A^{1/p^\infty}$, we find that $\widehat{T}$ is indeed the $p^\flat$-adic completion of $T$. By \cite[Lemma 3.3]{NS23}, $\widehat{T}$ is integrally closed in $\widehat{T}[1/p^\flat]$. Since $(R_\infty)^\flat$ is a normal domain and $(R_\infty)^\flat[1/p^\flat] \to \widehat{T}[1/p^\flat]$ is finite \'etale, it follows that $\widehat{T}$ is a normal ring, that is, it is a finite product of normal domains. What remains to be proved is that $\widehat{T}$ is an integral domain. This fact can be checked by using Heitmann's result. Note that since $T$ is a normal domain, the integral extension $T \hookrightarrow A^+$, where $A^+$ is the absolute integral closure of $A$, induces an injection $T/p^n T \hookrightarrow A^+/p^n A^+$ for all $n>0$. Taking the inverse limit, the injection follows: $\widehat{T} \hookrightarrow \widehat{A^+}$. By the main result of \cite{Hei22}, we conclude that $\widehat{T}$ is a domain. In other words, $\widehat{T}$ is a normal domain that is a subring of $(R_{\infty,p})^\flat$. Note that $C=\varinjlim T$, where $A^{1/p^\infty} \to T$ ranges over all $(p^\flat)^{1/p^\infty}$-almost finite \'etale extensions contained in $A^+$ such that $T$ is a normal domain. Then this gives us $C \hookrightarrow B$. On the other hand, the above argument together with an equivalence of $\Phi_1$ gives that $B \hookrightarrow \varinjlim \widehat{T} \hookrightarrow \widehat{C}$, where the completion is $p^\flat$-adic. Hence we conclude from the assertion $(2)$ that $\widehat{C} \cong \widehat{B} \cong (R_{\infty,p})^\flat$, as desired.
\end{proof}

\begin{corollary}
\label{Galoistilting2}
Let the assumptions be as in Theorem \ref{Galoistilting}. Let $\widehat{R_\infty} \to \widehat{R_{\infty,p}}$ be the $p$-adic completin of $R_\infty \to R_{\infty,p}$. Let $C$ be the unique maximal integral extension over $\widehat{R_\infty}$ such that $C$ is an integrally closed domain which is a subring of the integral domain $\widehat{R_{\infty,p}}$ and the localization map $\widehat{R_\infty}[1/p] \to C[1/p]$ is a filtered colimit of all finite \'etale $\widehat{R_\infty}[1/p]$-algebras contained in $\widehat{R_{\infty,p}}[1/p]$. Then $C$ is integral over $\widehat{R_{\infty}}$ and $\widehat{R_{\infty,p}}$ is equal to the $p$-adic completion of $C$.
\end{corollary}

\begin{proof}
The proof of Theorem \ref{Galoistilting} works in the present situation by making symbolic and minor changes. So we omit the details.
\end{proof}

\end{document}